\documentclass[10pt]{amsart}
\usepackage{amssymb, amscd, amsmath, amsthm, latexsym, enumerate}
\newtheorem{theorem}{Theorem}
\newtheorem{lemma}[theorem]{Lemma}

\newtheorem*{cor}{Corollary}
\newtheorem*{thm}{Theorem}

\begin{document}

\title{Seifert fibred 2-knot manifolds. II}

\author{Jonathan A. Hillman }
\address{School of Mathematics and Statistics\\
     University of Sydney, NSW 2006\\
      Australia }
\email{jonathan.hillman@sydney.edu.au}

\begin{abstract}
We show that if $B$ is an aspherical 2-orbifold in one of the families known to have orbifold fundamental groups of weight 1 then $B$ is the base of a Seifert fibration of a 2-knot manifold $M(K)$.
\end{abstract}



\maketitle
A 4-manifold $M$ is Seifert fibred over a 2-orbifold base $B$
if there is an orbifold fibration $p:M\to{B}$ with general fibre a torus $T$.
The knot manifold $M(K)$ associated to an $n$-knot $K$ is the closed $(n+2)$-manifold obtained by elementary surgery on the knot.
(When $n=1$ we assume the surgery is 0-framed.)

This note is a continuation of the paper \cite{HH13},
in which the work of the second author on the Scott-Wiegold conjecture \cite{Ho}
was used to constrain the possible orbifold bases 
of Seifert fibrations of knot manifolds.
The constraints were summarised in \cite[Theorem 4.2]{HH13},
which asserts the following:

\begin{thm}
Let $K$ be a $2$-knot with group $\pi=\pi{K}$,
and such that the knot manifold
$M=M(K)$ is Seifert fibred, with base orbifold $B$.
If $\pi'$ is infinite then $M$ is aspherical and $B$ is either
 \begin{enumerate}
\item$S^2(a_1,\dots,a_m)$, with $m\geq3$, 
no three of the cone point orders $a_i$ have a nontrivial common factor, 
at most two disjoint pairs each have a common factor, and trivial monodromy; or
\item$P^2(b_1,\dots,b_n)$ with $n=2$ or $3$,
the cone point orders $b_i$ being pairwise relatively prime, and $b_1=2$ if $n=3$,
or  $P^2(3,4,5)$,
and monodromy of order $2$ and non-diagonalizable; or
\item$\mathbb{D}(c_1\dots,c_p,\overline{d_1},\dots,\overline{d_q})$,
with the cone point orders $c_i$ all odd and relatively prime,
and at most one of the $d_j$ even,  $p\leq2$ and $2p+q\geq3$, 
and monodromy of order $2$ and diagonalizable.
\end{enumerate}
All such knot manifolds are mapping tori,
and are geometric.
\qed
\end{thm}

The lower bounds are to ensure asphericity.
In case (3), we should also exclude $\mathbb{D}(3,\overline2)$
and $\mathbb{D}(\overline2,\overline3,\overline5)$,
which are quotients of $S^2$.
All these groups are the groups of knots 
in homology 4-spheres.

Our concern here is with showing that the 2-orbifolds known 
to have groups of weight 1 are bases for Seifert fibrations 
of 2-knot manifolds.
(The corresponding problem for 1-knots remains open.)
We shall limit our consideration of the question 
``when does $G=\pi^{orb}B$ have weight 1?" to outlining in \S4 why 
we think that most of the possible Seifert bases are known.

\section{some general results}

We shall show here that an aspherical 2-knot manifold 
is Seifert fibred if and only if it is the total space 
of an $S^1$-bundle over a Seifert fibred 3-manifold.

\begin{lemma}
Let $\nu$ be a group with $\nu/\nu'\cong\mathbb{Z}$
and with a central subgroup $\kappa\cong\mathbb{Z}$
such that $G=\nu/\kappa$ has weight $1$. 
If $G/G'\cong\mathbb{Z}/e\mathbb{Z}$, where $e\leq4$ or $e=6$, 
then $\nu$ has weight $1$.
\end{lemma}

\begin{proof}
Let $t,u\in\nu$ represent a generator for $\nu/\nu'$ and 
a weight element for $G$, repectively,
and let $C=\langle\langle{u}\rangle\rangle$ 
be the normal closure of $u$ in $\nu$.
Then $\nu=C\kappa$.
If $e=0$ then $H_1(p;\mathbb{Z})$ is an isomorphism,
and so $\nu/C$ has trivial abelianization.
On the other hand, $\nu/C$ is a quotient of $\kappa$, 
and so is abelian.
Hence it is trivial, and so $\nu$ has weight 1.

If $e\not=0$ then $\nu'\cap\kappa=1$, 
so  $\nu'\leq{C}$.
Since $u$ normally generates $G$, 
$u\in{t^r}\nu'$ for some $r$ such that $(e,r)=1$,
while $\kappa$ is generated by some $k\in{t^e}\nu'$.
The hypothesis on $e$ implies that we may write $r=se\pm1$ for some $s$.
Let $v=uk^{-s}$. Then $v$ is a weight element for $\nu$.
\end{proof}

We do not know whether the restriction on $e$ is necessary.

A similar argument shows that an extension $E$ of a group $G$ 
by a solvable normal subgroup $S$ has weight $1$ if $E/E'\cong{G/G'}$ 
and $G$ has weight $1$.

\begin{theorem}
Let $M$ be an aspherical closed orientable $4$-manifold with $H_1(M;\mathbb{Z})\cong\mathbb{Z}$.
Then $M$ is Seifert fibred with general fibre a torus 
if and only if $M$ is the total space of an $S^1$-bundle 
over an aspherical Seifert fibred $3$-manifold $N$ with
$H_1(N;\mathbb{Z})\cong\mathbb{Z}$ and Euler class a generator of 
$H^2(N;\mathbb{Z}^{w_1(N)})$.
\end{theorem}

\begin{proof}
Suppose first that $p:M\to{B}$ is a Seifert fibration 
with general fibre a torus.
Let $\pi=\pi_1(M)$ and let $A$  
be the image of the fundamental group of a general fibre of $p$.
Then $A\cong\mathbb{Z}^2$,
since $M$ is aspherical, $A$ is normal in $\pi$ and
$G=\pi/A\cong\pi^{orb}B$.
Since $G$ has weight 1 the abelianization $G/G'$ is finite, 
and so $A\cap\pi'\cong\mathbb{Z}$. 
Let $\sigma=\pi/A\cap\pi'$
and let $w:\sigma=\pi/A\cap\pi'\to\mathbb{Z}^\times$ 
be the homomorphism corresponding to the action of $\pi$ on $A\cap\pi'$
by conjugation.
Then $\sigma/\sigma'\cong\mathbb{Z}$,
and $\sigma$ is a central extension of $G$ by $\mathbb{Z}$.

Suppose $g\in\sigma$ has finite order. Then 
the image of $g$ in $G$ has finite order also, 
and so is conjugate to a power of a generator associated to a cone point or a reflector curve or corner point of $B$, by Theorem 4.8.1 of \cite{ZVC}.
But then either $g\in{A}\cap\pi'$ or
the image of $g$ is not in $G'$,
and so $g\not\in\sigma'$.
Hence $\sigma$ is torsion free,
and so it is the fundamental group of an
aspherical Seifert fibred 3-manifold $N$.
The projection $p$ of $M$ onto $B$ factors through a submersion onto $N$ with fibre $S^1$, and so $M$ is the total space of an $S^1$-bundle over $N$.

It follows from the Gysin sequence for $\pi$ as an extension of $\sigma$
by $\mathbb{Z}$ and Poincar\'e duality for $N$ that
$N$ is an homology handle,
the action of $\sigma$ on $A\cap\pi'\cong\mathbb{Z}$ is given by $w_1(N)$
and the Euler class of the extension generates $H^2(N;\mathbb{Z}^{w_1(N)})$.

Conversely, if $q:M\to{N}$ is the projection of an $S^1$-bundle over 
an aspherical $3$-manifold $N$ with
$H_1(N;\mathbb{Z})\cong\mathbb{Z}$ and Euler class a generator of 
$H^2(N;\mathbb{Z}^{w_1(N)})$ then $M$ is aspherical,
orientable and $H_1(M;\mathbb{Z})\cong\mathbb{Z}$.
If $r:N\to{B}$ is a Seifert fibration then the composite $p=rq$ is a Seifert fibration with general fibre a torus.
\end{proof}

The 3-manifold $N$ need not be orientable, 
as there are examples whose base orbifolds have reflector curves.
The correspondance between the 4-manifolds $M$ and 3-manifolds $N$ 
is bijective.

\begin{cor}
Let $M$ be as above.
Then $M\cong{M(K)}$ for some $2$-knot $K$ 
if and only if $\pi_1(N)$ has weight $1$.
\end{cor}

\begin{proof}
If $M\cong{M(K)}$ then $\pi$ has weight 1, 
and so $\nu=\pi/A\cap\pi'$ has weight 1.
Conversely, if $\nu$ has weight 1 then so does $\pi$.
The result of surgery on a simple closed curve in $M$
representing a normal generator of $\pi$ is a 1-connected 4-manifold 
$\Sigma$ with  $\chi(\Sigma)=2$.
Hence $\Sigma$ is homeomorphic to $S^4$, and the cocore of
the surgery is a 2-knot $K$ with $M(K)\cong{M}$.
\end{proof}

Thus it shall suffice to show that each of the orbifolds $B$ under
consideration is the base of a Seifert fibration of a 3-manifold
with fundamental group having abelianization $\mathbb{Z}$ and weight 1.

\section{examples}

When the numbers of cone points are small the groups $G=\pi^{orb}B$ corresponding to the 2-orbifolds in \cite[Theorem 4.2]{HH13} have weight 1.
This is so in case (1) when $m=3$, 
or when $m=4$ and $(a_1,a_2)=(a_3,a_4)=1$,
in case (2) when $n=2$, and in case (3) when $p=2$ and $q=0$,
or when $p\leq1$.
In case (3), if $p=0$ all such orbifolds are the bases of Seifert fibrations of $M(K)$, for $K$ the 2-twist spin of some Monesinos knot.
We shall show that in each of the other subcases just listed $B$ 
is the base of a Seifert fibration of an aspherical 3-manifold,
and so is the base of Seifert fibrations of 2-knot manifold
$M(K)$ for some 2-knot $K$, by the corollary to Theorem 2.
We would also like to know which extensions of $G$ by $\mathbb{Z}$ 
have weight 1,
but we are limited by our ignorance of the possible weight elements 
for such $G$.

All the orbifolds considered below are hyperbolic,
with the exception of $S^2(2,3,6)$, which is flat,
and Seifert fibred 3-manifolds with such bases are aspherical.

A finitely generated group has infinite cyclic abelianization 
if and only if the first two elementary ideals of the matrix of exponents 
in the relators are $E_0=0$ and $E_1=(1)$.

\medskip
$S^2(p,q,r)$.
If a 3-manifold $N$ is Seifert fibred over $S^2(p,q,r)$ then
$\nu=\pi_1(N)$ is torsion free and has a presentation
\[
\langle{t,u,v,w}\mid{u^p=t^a},~v^q=t^b,~w^r=t^c,~uvw=t^d,~t~central
\rangle
\]
We may assume that $d=0$, after replacing $z$ by $zt^{-d}$, if necessary.
It is then easy to see that 
\[
E_0=(aqr+bpr+cpq)\quad\mathrm{and}\quad{E_1=(pq,pr,qr,aq,ar,bp,br,cp,cq)}.
\]
Since $\nu$ is torsion free, $(a,p)=(b,q)=(c,r)=1$,
and so $E_1=(p,q,r)$.
If also $\nu/\nu'\cong\mathbb{Z}$ then
$aqr+bpr+cpq=0$ and $(p,q,r)=1$.
These conditions imply that $p|qr$, $q|pr$ and $r|pq$.
We may write $p=p_qp_r$, $q=p_qq_r$ and $r=\pm{p_r}q_r$, 
where $(p_q,p_r)=(p_q,q_r)=(p_r,q_r)=1$.
The condition $E_0=0$ becomes $aq_r+bp_r+cp_q=0$.

If these conditions hold, 
it is easy to see that $u^{-1}v$ represents 
a weight element for the orbifold group $G$;
it represents a weight element for $\nu$ 
if and only if $(aq-bp,2ar+cp,2br+cq)=1$.

Given $p,q,r$ such that $(p,q,r)=1$, while $p|qr$, $q|pr$ and $r|pq$,
we may assume that $p$ is odd.
Let $a_o,b_o$ be any solution of $aq_r-bp_r=1$.
Since $(p_q,2p_rq_r)=1$, there is an $x$ such that
$a_oq_r+b_op_r+2xp_rq_r\equiv{p_q}(p_q+1)$ {\it mod} $(p_q^2)$.
Let $a=a_o+xp_r$, $b=b_o+xq_r$ and $c=-\frac{a_oq_r+b_op_r+2xp_rq_r}{p_q}$.
Then $aq_r+bp_r+cp_q=0$,
$aq_r-bp_r=1$ and $aq_r+bp_r\equiv0$ {\it mod} $(p_q)$,
and so $(a,p)=(b,q)=(c,r)=1$.
Moreover  $(aq-bp,2ar+cp,2br+cq)=1$.
Adjoining the relation $u=v$ gives $w=u^{-2}$, and the presentation 
collapses to $\langle{t,u}\mid{u^p=t^a},~u^q=t^{-b},~u^{2r}=t^{-c},~tu=ut\rangle$,
which gives an abelian group of order $(aq-bp,2ar+cp,2br+cq)=1$.
Thus $u^{-1}v$ is a weight element for the extension
corresponding to this choice of exponents $a,b,c$.

\medskip
$S^2(p,q,r,s)$ [with $[(p,q)=(t,s)=1$].
If a 3-manifold $N$ is Seifert fibred over $S^2(p,q,r,s)$ then
$\nu=\pi_1(N)$ is torsion free and has a presentation
\[
\langle{t,u,v,w,x}\mid{u^p=t^a},~v^q=t^b,~w^r=t^c,~x^s=t^d,~
uvwx=t^e,~t~central
\rangle
\]
We may assume that $e=0$.
We then have 
\[
E_0=(aqrs+bprs+cpqs+dpqr)
\]
and 
\[E_1=(pqr,pqs,prs,qrs,aqr,aqs,ars,bpr,bps,brs,
\]
\[cpq,cps,cqs,dpq,dpr,dqr).
\]
Since $\nu$ is torsion free,
$(a,p)=(b,q)=(c,r)=(d,s)=1$,
and so $E_1=(pq,pr,ps,qr,qs,rs)$.

If, moreover, $(p,q)=(r,s)=1$ then we may write $p=p_rp_s$ and $q=q_rq_s$, 
with $r=\epsilon_rp_rq_r$
and $s=\epsilon_sp_sq_s$, where $|\epsilon_r|=|\epsilon_s|=1$.
The condition $E_0=0$ then becomes
$\epsilon_r\epsilon_s(aq+bp)+cs+dr=0$,
while $E_1=(1)$.
Given such $p,q,r,s$, we may choose $a,b,c,d$ so that $aq+bp=1$
and $cs+dr=-\epsilon_r\epsilon_s$.
The image of $uv$ is a weight element 
for the corresponding extension.

\medskip
$P^2(p,q)$.
If a 3-manifold $N$ is Seifert fibred over $P^2(p,q)$ then
$\nu=\pi_1(N)$ is torsion free and has a presentation
\[
\langle{t,u,v,w}\mid{u^2=vwt^k},~v^p=t^e,~w^q=t^f,~t~central
\rangle
\]
We may assume that $k=0$, after replacing $w$ by $wt^k$, if necessary.
Clearly $E_0=0$, while $E_1=(2pq,2pf,2qe,pf+qe)$.
Thus if $\nu/\nu'\cong\mathbb{Z}$ then
$\nu$  is torsion free if and only if $(p,e)=(q,f)=1$,
and then we must have $E_1=(2, pf+qe)=1$.
Hence $(p,q)=1$ and either $p$ or $q$ is even or $e+f$ is odd.

Given $(p,q)=1$, we may choose $e,f$ such that $pf-qe=1$.
In this case $u^{-1}v$ represents a weight element 
for the corresponding extension.

\medskip
$\mathbb{D}(a,\overline{d_1},\dots,\overline{d_q})$.
If a 3-manifold $N$ is Seifert fibred over $\mathbb{D}(a,\overline{d_1},\dots,\overline{d_q})$ then
$\nu=\pi_1(N)$ is torsion free and has a presentation
\[
\langle{t,u,x_1,\dots,x_{q+1}}\mid{u^a=t^e},~x_i^2=t^{g_i},
~(x_ix_{i+1})^{d_i}=t^{f_i},~\forall~i\leq{q},
\]
\[x_{q+1}u=ux_1t^h,~t~central\rangle
\]
We may assume that $g_i=1$, for all $i\leq{q}$, and $h=0$.
If $\nu/\nu'\cong\mathbb{Z}$ then $a$ is odd, and if also
$\nu$ is torsion free then $(a,e)=1$ and $(d_i,f_i)=1$,
for all $i\leq{q}$.

Given an odd integer $a$, if $e=\frac12(a-1)$ then $a-2e=1$,
and so adjoining the relation $u=x_1$ trivializes the group.
In this case $u^{-1}x_1$ represents a weight element 
for the corresponding extension.

\medskip
$\mathbb{D}(a,b)$.
If a 3-manifold $N$ is Seifert fibred over $\mathbb{D}(a,b)$ then
$\nu=\pi_1(N)$ is torsion free and has a presentation
\[
\langle{t,u,v,w}\mid{u^a=t^e},~v^b=t^f,~w^2=t^g,~
wuv=uvwt^h,~t~central\rangle
\]
We may assume that $g=1$ and $h=0$.
If this group has weight 1 then $a,b$ are odd and $(a,b)=1$.
Adjoining the relation $u=wv$ gives an abelian group of order 
$2(af-be)+ab$.

Given $a,b$ odd and $(a,b)=1$, 
we may choose $e,f$ such that $2(af-be)=1-ab$.
In this case $u^{-1}wv$ represents a weight element 
for the corresponding extension.

\section{twist spins}

The knot manifolds of branched twist spins of classical knots
are Seifert fibred, and their groups have weight elements 
such that some power is central. 
If $B=S^2(a_1,\dots,a_m)$ is the base of a Seifert fibration of a
$M(K)$ for some $r$-twist spin $K$ then it follows from  
\cite[Theorem 4.8.1]{ZVC} that $m=3$, and
$B=S^2(p,q,r)$ with $(p,q)=1$.
(See \cite[Lemma 16.7]{Hi}.)
A similar argument shows that if  $\mathbb{D}(c_1\dots,c_p,\overline{d_1},\dots,\overline{d_q})$
and $G=\pi^{orb}B$ has a weight element of finite order $r$ 
then $p=0$ and $r=2$.
The orbifolds $S^2(p,q,r)$ with $(p,q)=1$ and
$\mathbb{D}(\overline{d_1},\dots,\overline{d_q})$
are realized by $r$-twist spins of $(p,q)$-torus knots
and 2-twist spins of Montesinos knots, respectively.

On the other hand, if $B=P^2(b_1,\dots,b_n)$ for some $n\geq2$
then all elements of finite order in $G=\pi^{orb}B$ have order 
dividing one of the exponents $b_i$.
Since $G/G'$ has order $2\Pi{b_i}$, 
there are no weight elements of finite order,
and so these orbifolds are not realized by twist spins.
Are any of them realized by other constructions based on classical knots?

In general, we do not know whether orientable Seifert fibred 3-manifolds $N$
with $H_1(N;\mathbb{Z})\cong\mathbb{Z}$  and $\pi_1(N)$ having weight 1 are
knot manifolds.

\section{the unsettled cases}

If a group has weight 1 then so does every quotient.
Thus to show that the groups considered in \S2 above exhaust the 2-orbifold groups of weight 1 it would suffice to eliminate certain special cases.

Each of the groups $\pi^{orb}S^2(a_1,\dots,a_m)$ (with $m\geq5$) or
$\pi^{orb}\mathbb{D}(c_1\dots,c_p,\overline{d_1},\dots,\overline{d_q})$
(with $q\geq1$) maps onto $\pi^{orb}S^2(a_1,\dots,a_5)$ or
$\pi^{orb}\mathbb{D}(c_1\dots,c_p,\overline{d_1})$,
respectively.
If $\{a_1,a_2,a_3,a_4\}$ is a quadruple such that 
no three terms have a common factor,
but with no partition into two pairs of relatively prime integers 
then $\pi^{orb}S^2(a_1,a_2,a_3,a_4)$ has a quotient of one of the forms 
$\pi^{orb}S^2(a,b,c,abc)$ or $\pi^{orb}S^2(bc,ac,ab,d)$,
where $a,b,c,d$ are distinct primes.

Thus we may concentrate on the following cases:
\begin{enumerate}
\item$S^2(a_1,\dots,a_5)$, with 
no three of the $a_i$ having a nontrivial common factor, 
and at most two disjoint pairs having a common factor;
\item$S^2(a,b,c,abc)$, with $a,b,c$ distinct primes;
\item$S^2(bc,ac,ab,d)$, with $a,b,c,d$ distinct primes;
\item$P^2(2,b,c)$, with $b,c$ odd and $(b,c)=1$, or $P^2(3,4,5)$;
\item$\mathbb{D}^2(a,b,\overline{d})$, with $(a,b)=1$.
\end{enumerate}
In cases (4) and (5) we could assume further that $a,b$ and $c$
are distinct primes.

In each of the cases (2), (3) and (5) the group $G=\pi^{orb}B$ may be obtained from a group of weight 2 by adjoining one relation:
\[
\pi^{orb}S^2(a,b,c,abc)\cong|\langle{v,w,x}\mid {v^a=w^b=x^c=1}\rangle|/\langle\langle(vwx)^{abc}\rangle\rangle,
\]
\[
\pi^{orb}S^2(bc,ac,ab,d)\cong|\langle{v,w,x}\mid {v^{ac}=w^{ab}=x^d=1}\rangle|/\langle\langle(vwx)^{bc}\rangle\rangle,
\]
and
\[
\pi^{orb}\mathbb{D}^2(a,b,\overline{d})\cong|\langle{v,w,x}\mid {u^a=v^b=x^2=1}\rangle|/\langle\langle[x,uv]^d\rangle\rangle.
\]
Thus showing that $G$ does not have weight 1 might be quite difficult.

On the other hand,
a natural extension of the Scott-Wiegold conjecture (raised in \cite{Ho})
is that the free product of $2k+1$ cyclic groups
should have weight $>k$.
If this is so, then $\pi^{orb}S^2(a_1,\dots,a_m)$ has weight $>1$ if $m\geq5$.
Similarly, the hypothesis ``$w(G*\mathbb{Z})>w(G)$"
is an extension of the Kervaire-Laudenbach Conjecture.
If $a,b,c$ are pairwise relatively prime then the group  $G=\mathbb{Z}/a\mathbb{Z}*\mathbb{Z}/b\mathbb{Z}*\mathbb{Z}/c\mathbb{Z}$
with presentation $\langle{v,w,x}|v^a=w^b=x^c=1\rangle$ has weight $2$.
Since $\pi^{orb}P^2(a,b,c)$ may be obtained from $G$ by adjoining 
one generator $u$ and one relation ($u^2=vwx$),
it has weight at least $w(G*\mathbb{Z})-1$.
If the extended Kervaire-Laudenbach Conjecture holds 
then $\pi^{orb}P^2(a,b,c)$ has weight 2.
Thus it is quite plausible that none of the 2-orbifolds
$S^2(a_1,\dots,a_5)$ or $P^2(a,b,c)$
may be realizable in our context.

\smallskip
{\it Acknowledgment.} This work was begun as a consequence of 
the BIRS conference on ``Unifying Knot Theory in Dimension 4", 
4-8 November 2019.


\end{document}